\documentclass[colorlinks,11pt,reqno]{amsart}

\usepackage[T1]{fontenc}
\usepackage[utf8]{inputenc}
\usepackage{verbatim}
\usepackage{times}
\usepackage{stmaryrd}
\usepackage[a4paper]{geometry}
\usepackage{graphicx}
\usepackage{tabls}
\usepackage{url}

\newtheorem{theorem}{Theorem}[section]
\newtheorem{lemma}[theorem]{Lemma}
\newtheorem{proposition}[theorem]{Proposition}
\newtheorem{Thm}{Theorem}

\theoremstyle{definition}
\newtheorem{definition}[theorem]{Definition}

\newtheorem{corollary}[theorem]{Corollary}

\theoremstyle{remark}
\newtheorem{remark}[theorem]{Remark}

\numberwithin{equation}{section}

\usepackage{ulem}
\usepackage{amsrefs}
\usepackage{color}

\let\oldmarginpar\marginpar
\renewcommand\marginpar[1]{\-\oldmarginpar[\raggedleft\footnotesize #1]
{\raggedright\footnotesize #1}}

\usepackage[T1]{fontenc}
\usepackage{lmodern}
\usepackage{setspace}
\usepackage{indentfirst}

\usepackage{multicol}
\usepackage[linktocpage = true]{hyperref}
\hypersetup{
 colorlinks = true,
 citecolor = blue,
}
\usepackage{color}
\definecolor{red}{rgb}{1,0,0}
\definecolor{green}{rgb}{0,1,0}
\definecolor{blue}{rgb}{0,0,1}
\definecolor{refkey}{gray}{.625}
\definecolor{labelkey}{gray}{.625}

\usepackage{amsmath,amssymb}
\usepackage{amsfonts}
\usepackage{amsthm}
\usepackage{amscd}
\usepackage{enumerate}
\usepackage{paralist}
\usepackage{graphicx}
\usepackage{mathtools}
\usepackage{kantlipsum}
\usepackage{tikz-cd}
\usetikzlibrary{arrows.meta,calc,positioning}
\allowdisplaybreaks

\numberwithin{equation}{section}

\newcommand{\RR}{\mathbb{R}}
\newcommand{\CC}{\mathbb{C}}

\newcommand{\cE}{\mathcal{E}}

\newcommand{\Pair}[2]{\left\langle #1,#2\right\rangle}

\newcommand{\C}{\mathbb{C}}
\newcommand{\CP}{\mathbb{CP}}
\newcommand{\R}{\mathbb{R}}
\newcommand{\Z}{\mathbb{Z}}

\DeclareMathOperator{\Spec}{Spec}

\DeclareMathOperator{\Hom}{Hom}
\DeclareMathOperator{\Span}{Span}
\DeclareMathOperator{\relint}{relint}
\DeclareMathOperator{\codim}{codim}
\DeclareMathOperator{\charac}{char}

\begin{document}

\title{Holomorphic polyvector fields on toric varieties via Klyachko filtrations}

\author{Wei Hong}
\address{School of Mathematics and Statistics, Wuhan University, Wuhan, 430072, China}
\email{\href{mailto:~~~hong\textunderscore  w@whu.edu.cn}{hong\textunderscore  w@whu.edu.cn}}

\author{Maosong Xiang}
\address{School of Mathematics and Statistics, Center for Mathematical Sciences, Huazhong University of Science and Technology, Wuhan, China.}
\email{\href{mailto:~~~msxiang@hust.edu.cn}{msxiang@hust.edu.cn}}

\begin{abstract}
Motivated by the extended deformation theory of complex manifolds, we give a combinatorial description of holomorphic polyvector fields on a smooth compact toric variety via Klyachko's filtrations.
This yields dimension and formal equivariant-character formulas and recovers the descriptions of Demazure roots and anticanonical sections.
\end{abstract}

\maketitle

\noindent  {\it Keywords:} \hspace*{0.6cm}
Toric varieties, polyvector fields, Klyachko filtrations, anticanonical polytopes. \\
\noindent  {\it MSC(2020):}\hspace*{0.5cm}  primary 14M25; secondary 14F05, 17B66.

\tableofcontents

\section{Introduction}\label{sect-Intr}

Let \(X\) be a smooth complex variety. Its global holomorphic polyvector fields form the graded vector space
\[
H^0\!\left(X,\bigwedge\nolimits^\bullet T_X\right) = \bigoplus_{k\geq0}
H^0\!\left(X,\bigwedge\nolimits^kT_X\right).
\]
This space carries a Gerstenhaber algebra structure under the wedge product and the Schouten--Nijenhuis bracket. Polyvector fields occur naturally in deformation quantization, extended deformation theory and mirror symmetry; see~\cites{BCOV, B-K 98, Kon}.

For an $n$-dimensional toric variety $X_\Delta$, the torus action makes the tangent bundle and its exterior powers equivariant. Their spaces of global sections therefore decompose into torus-weight spaces.
In degree one, Demazure described the nonzero weights of \(H^0(X_\Delta, T_{X_\Delta})\) in terms
of the lattice points now called Demazure roots~\cite{Demazure}.
In top degree $n$,
\[
\bigwedge\nolimits^n T_{X_\Delta} \cong \mathcal{O}_{X_{\Delta}}(-K_{X_{\Delta}}),
\]
and the weight decomposition is the standard lattice-point description of sections of a torus-invariant divisor~\cites{Cox, Fulton}.

The purpose of this article is to give a uniform and self-contained account of
\[
H^0\!\left(X_\Delta,\bigwedge\nolimits^kT_{X_\Delta} \right), \qquad 0\leq k\leq n,
\]
for a smooth complete toric variety \(X_\Delta\) of dimension $n$.
The proof is formulated in terms of Klyachko's description of equivariant toric vector bundles by compatible filtrations \cite{Klyachko}.

To state our main theorem, we fix some notations:
Let \(N \cong \Z^n\) be a lattice, let
\[
M=\Hom(N, \Z)
\]
be its dual.
Let \(\Delta\) be a smooth complete fan in
\[
N_{\R}=N \otimes_{\Z} \RR.
\]
For every ray \(\rho \in \Delta(1)\), denote by \(v_\rho\in N\) its primitive generator.

Define
\[
P_\Delta =
\left\{
u\in M_{\R}\ \middle|\
\Pair{u}{v_\rho}\geq-1
\text{ for every }\rho\in\Delta(1)
\right\}.
\]
For \(u\in P_\Delta\), set
\[
A(u) =
\left\{
\rho\in\Delta(1)
\ \middle|\
\Pair{u}{v_\rho}=-1
\right\},
\]
and
\[
L_u =
\Span_{\C}
\{v_\rho\mid\rho\in A(u)\}
\subseteq N_{\C}.
\]
Put
\[
i(u)=\dim_{\C}L_u.
\]
For \(1\leq k\leq n\), define
\[
N_u^k
=
\begin{cases}
\displaystyle
\bigwedge\nolimits^{i(u)}L_u
\wedge
\bigwedge\nolimits^{k-i(u)}N_{\C},
&i(u)\leq k,\\[6pt]
0,&i(u)>k.
\end{cases}
\]
Note that the Lie algebra of the algebraic torus is isomorphic to $N_\C$. The infinitesimal torus action on $X$ induces a map
\[
\rho_N \colon N_{\C} \longrightarrow H^0(X_\Delta,T_{X_\Delta}),
\]
which is extended to exterior powers
\[
 \rho_N \colon \bigwedge\nolimits^k N_{\C} \longrightarrow H^0(X_\Delta,\bigwedge\nolimits^k T_{X_\Delta}).
\]
\begin{Thm}[=Theorem~\ref{thm:main} and Corollary~\ref{cor:dimension}] \label{General-multiVect-thm}
Let \(X_\Delta\) be an \(n\)-dimensional smooth complete toric
variety. For \(1\leq k\leq n\),
\[
H^0\!\left( X_\Delta,\bigwedge\nolimits^kT_{X_\Delta} \right)
= \bigoplus_{\substack{u\in P_\Delta\cap M\\i(u)\leq k}} \chi^u\rho_N(N_u^k),
\]
where $\chi^u$ is the corresponding character.

If \(u\) lies in the relative interior of a face of codimension \(i\), then
\[
i(u)=i
\]
and
\[
\dim_{\C}\chi^u\rho_N(N_u^k)
=
\binom{n-i}{k-i}.
\]
Consequently,
\[
\dim_{\C} H^0\!\left(X_\Delta,\bigwedge\nolimits^kT_{X_\Delta} \right) = \sum_{i=0}^{k}
\binom{n-i}{k-i} \sum_{\substack{F\preceq P_\Delta\\\codim F=i}}
\#\bigl(\relint(F)\cap M\bigr).
\]
\end{Thm}
In particular, when $k=1$,  lattice points in the relative interior of facets are identified with the Demazure roots of the Lie algebra $H^0(X_\Delta, T_{X_{\Delta}})$,
we thus recover Demazure's result~\cite{Demazure} on holomorphic vector fields on $X_{\Delta}$.
When $X_{\Delta} = \CP^n$, this theorem has already been proved by the first author in~\cite{Hong 19}.

Finally, we would like to mention some works that are related to the present paper.
The weight-space decomposition studied in this paper had previously appeared in two preprints by the first author \cite{HongPoisson, HongPolyvector}. It was subsequently stated as
Theorem~1.1 by Deng and Hong \cite{Deng-Hong 21}, with attribution to \cite{HongPolyvector}, and was applied there to the study of Batalin-Vilkovisky operators on the Gerstenhaber algebra of
holomorphic polyvector fields. The present paper provides a self-contained derivation of the decomposition from Klyachko's filtration formalism, makes the weight convention explicit, and identifies the resulting filtration intersections with the normal spaces of faces of the associated polytope.
In~\cite{Filip}, Filip gave a convex geometric description of the Hodge decomposition of the Hochschild cohomology of an affine toric variety, which can be used to prove the existence of deformation quantization of some possibly singular affine toric varieties. It might be interesting to connect our combinatorial description of polyvector fields (in the affine case) to Filip's description of Hochschild cohomology to understand deformation quantization in the combinatorial manner.
Recently, in order to compute multiplicities in tropical and log Gromov-Witten theory, Mandel and Ruddat~\cite{M-R 19} studied the space $A = \Z[N] \otimes \wedge^\bullet M$ of integral polyvector fields on the algebraic torus $\operatorname{Spec}\Z[N]$. They proved that this space admits a Batalin-Vilkovisky algebra structure and that a submodule $A_0$ of $A$ carries an $L_\infty$-algebra structure.

The paper is organized as follows.
In Section~\ref{sec:preliminaries} we introduce the toric notation, the anticanonical polytope, and the relevant Klyachko filtrations.
Then we prove Theorem~\ref{thm:main} in Section~\ref{sec:weights}.
In Section~\ref{subsec:formal-character}, we  prove a formal character formula, which provides a generating function reformulation of the weight-space decomposition.
Section~\ref{sec:affine} provides an affine interpretation of our main result.
Finally, in Section~\ref{sec:consequences}, we discuss two special cases on vector fields and anticanonical sections, respectively, and work out explicitly holomorphic vector fields on the blow-up of \(\mathbb P^3\) at a torus-fixed point.

\section{Preliminaries}\label{sec:preliminaries}

\subsection{Toric notation and the weight convention}

We use standard notation from toric geometry; see~\cites{Cox, Fulton}.

Let $N$ be a lattice of rank $n$, with dual lattice
\[
M=\Hom(N, \Z).
\]
The natural pairing is denoted by
\[
\langle -,- \rangle \colon M \times N\longrightarrow\Z.
\]
Set
\[
N_{\R}= N\otimes_{\Z}\R, \qquad N_{\C}=N\otimes_{\Z}\C,\qquad  M_{\R}=M\otimes_{\Z}\R.
\]
Given a smooth complete fan $\Delta$ in $N_{\R}$, let $X_\Delta$ be its associated toric variety. For each cone $\sigma \in \Delta$, denote by
\[
U_\sigma=\Spec\C [\sigma^\vee \cap M]
\]
the associated affine toric variety.
Note that the collection of affine open subsets $U_\sigma$ associated with $n$-dimensional cones $\sigma \in \Delta(n)$ is an open cover of $X_\Delta$.

The algebraic torus associated with $N$ is
\[
T= N\otimes_{\Z} \C^\ast  \cong (\C^\ast)^n.
\]
For each $u \in M$, let
\[
\chi^u\colon T \longrightarrow \C^\ast
\]
be the corresponding character, which can also be considered as a rational function on $X_\Delta$.

Since $X_\Delta$ is complete, it is proper. By GAGA, algebraic global sections of coherent sheaves agree with holomorphic global sections on the associated analytic variety. We therefore refer to
the elements of
\[
H^0\!\left(X_\Delta,\bigwedge\nolimits^kT_{X_\Delta}\right)
\]
as holomorphic polyvector fields.

The infinitesimal torus action induces an injection from the Lie algebra $N_\C$ of $T$ to the space of holomorphic vector fields
\[
\rho_N \colon N_{\C} \longrightarrow  H^0(X_\Delta,T_{X_\Delta})
\]
defined by
\[
\rho_N(x)(\chi^u) = \langle u, x \rangle \chi^u,
\]
for all $x \in N_\C$ and $u \in M$.
By abuse of notation, we denote its exterior extension by the same symbol:
\[
\rho_N \colon \bigwedge\nolimits^k N_{\C} \longrightarrow
H^0\!\left(X_\Delta, \bigwedge\nolimits^k T_{X_\Delta} \right).
\]
This exterior extension is also injective. Indeed, at every point of the dense torus, the infinitesimal action identifies $N_{\C}$ with the tangent space of the torus. Its $k$-th exterior power is therefore injective. Since a global section that vanishes on the dense torus vanishes identically, the displayed map is injective.

The algebraic torus $T$ also acts naturally on the section space $H^0\left(X_\Delta, \wedge^kT_{X_\Delta}\right)$ by
\begin{equation}\label{eq:section-action}
(t\cdot s)(p) = t_\ast \bigl(s(t^{-1}p)\bigr),
\end{equation}
for all $t \in T, s \in H^0\left(X_\Delta, \wedge^kT_{X_\Delta}\right), p \in X_\Delta$.
Note that the image $\rho_N(\wedge^k N_\C)$ is the space of $T$-invariant holomorphic $k$-vector fields.

For each $u \in M$, let
\[
H^0\!\left( X_\Delta,\bigwedge\nolimits^kT_{X_\Delta} \right)_u \subset H^0\left(X_\Delta, \wedge^kT_{X_\Delta}\right)
\]
be the weight-\(u\) subspace. If $w=\rho_N(\xi)$ for some $\xi \in \wedge^k N_\C$, then
\begin{equation}\label{eq:weight-sign}
t\cdot(\chi^u w) = \chi^{-u}(t)\chi^u w.
\end{equation}
Thus $\chi^u w$ has weight $-u$, i.e.,
\[
\chi^u\rho_N(\xi) \in H^0\!\left(X_\Delta, \bigwedge\nolimits^kT_{X_\Delta} \right)_{-u}.
\]
\begin{remark}\label{rem:weight-sign}
Some references use the opposite convention for the induced torus action on sections. This changes the labels of the weight spaces but not the underlying sections. We use~\eqref{eq:section-action} consistently, and hence retain the minus sign in \eqref{eq:weight-sign}.
\end{remark}

\subsection{The anticanonical polytope}
Let $X_\Delta$ be a smooth complete toric variety.
For every ray $\rho \in \Delta(1)$ in the fan $\Delta$, let $D_\rho$ be the corresponding invariant prime divisor. Then the anticanonical divisor is
\[
-K_{X_\Delta} = \sum_{\rho\in\Delta(1)} D_\rho.
\]
Its associated anticanonical polytope is
\begin{equation}\label{eq:polytope}
P_\Delta = \left\{ u\in M_{\R}\ \middle|\ \langle u, v_\rho \rangle \geq -1
\text{ for every }\rho \in \Delta(1)
\right\}.
\end{equation}

The following is a standard consequence of completeness of the fan
and the polytope construction for torus-invariant divisors.
\begin{proposition}[\cites{Cox, Fulton}]\label{lem:interior-origin}
The set $P_\Delta$ is a full-dimensional compact rational polytope.
Moreover, the origin $0$ is the only lattice point in the interior of \(P_\Delta\).
\end{proposition}
\begin{proof}
The set \(P_\Delta\) is an intersection of finitely many rational closed half-spaces. Since
\[
\langle 0, v_\rho\rangle = 0 >-1
\]
for every ray $\rho \in \Delta(1)$, the origin $0$ belongs to its interior.

Suppose that \(P_\Delta\) were unbounded. Its recession cone would contain a nonzero element \(a\in M_{\RR}\) satisfying
\[
\Pair{a}{v_\rho}\geq 0 \qquad \text{for every }\rho\in\Delta(1).
\]
Since \(\Delta\) is complete, every \(x\in N_{\R}\) belongs to a cone of \(\Delta\), and hence is a nonnegative linear combination of primitive ray generators. It follows that
\[
\Pair{a}{x} \geq 0 \qquad  \text{for every }x\in N_{\R}.
\]
Applying this to \(x\) and \(-x\) gives \(a=0\), a contradiction.
Thus \(P_\Delta\) is bounded.

Finally, for any lattice point \(u\in M\cap\operatorname{int}(P_\Delta)\) in the interior $\operatorname{int}(P_\Delta)$ of $P_\Delta$, we have
\[
\Pair{u}{v_\rho}>-1
\]
for all \(\rho\in\Delta(1)\). Since the pairings are integral, it follows that
\[
\Pair{u}{v_\rho} \geq 0.
\]
Since the fan $\Delta$ is complete, every $x \in N_\R$ is a nonnegative linear combination of primitive ray generators, which implies that
\[
  \Pair{u}{x} \geq 0.
\]
Applying this relation to both $x$ and $-x$ implies that \(u=0\).
\end{proof}

For \(u\in P_\Delta\), define
\[
A(u) = \left\{ \rho\in\Delta(1) \ \middle|\   \Pair{u}{v_\rho}=-1 \right\}
\]
and
\[
L_u = \Span_{\C}\{v_\rho\mid\rho\in A(u)\}.
\]
The next lemma is the standard active-inequality description of the normal space of a face of a polytope; see, for example, \cite{Ziegler}*{Chapter 2}, especially the discussion of faces, supporting hyperplanes, and active inequalities.

\begin{lemma}\label{lem:normal-space}
Given a face \(F\preceq P_\Delta\) of $P_\Delta$ and a point \(u\in\relint(F)\) of its relative interior, the vector space \(L_u\) is the complexification of the normal space of \(F\). In particular,
\[
\dim_{\C} L_u= \codim F.
\]
\end{lemma}
\begin{proof}
At a point in the relative interior of \(F\), the active inequalities in \eqref{eq:polytope} are precisely the supporting equalities whose supporting hyperplanes contain \(F\). Their normal
vectors are the vectors \(v_\rho\) with \(\rho\in A(u)\). They therefore span the normal space of \(F\), whose dimension is \(\codim F\).
\end{proof}

\subsection{Klyachko filtrations}\label{subsec:klyachko}

In~\cite{Klyachko}, Klyachko classified equivariant vector
bundles on toric varieties in terms of compatible filtrations. The tangent-bundle filtration is a standard example, and the exterior-power filtration is obtained by the usual
functorial construction \cites{Gonzalez, Klyachko}.

Let \(\cE\) be a \(T\)-equivariant vector bundle on \(X_\Delta\), and let \(E\) be its fiber over the identity of the dense torus. For every ray \(\rho\in\Delta(1)\), Klyachko associates a decreasing filtration
\[
\cdots\supseteq E^\rho(j) \supseteq E^\rho(j+1) \supseteq\cdots,   \qquad j\in \Z,
\]
satisfying a compatibility condition on each cone~\cite{Klyachko}.

Under the convention of Remark~\ref{rem:weight-sign}, the global-section formula takes the following form.
\begin{theorem}[Klyachko~\cite{Klyachko}*{Corollary 4.1.3}]\label{thm:klyachko}
Let \(\cE\) be a \(T\)-equivariant vector bundle on a complete toric
variety \(X_\Delta\), with Klyachko data
\[
\left(E,\{E^\rho(j)\}_{\rho,j}\right).
\]
For every \(u\in M\),
\begin{equation}\label{eq:klyachko}
H^0(X_\Delta,\cE)_{-u} \cong \bigcap_{\rho\in\Delta(1)}
E^\rho\!\left(-\Pair{u}{v_\rho}\right).
\end{equation}
\end{theorem}

\begin{remark}
The sign in \eqref{eq:klyachko} is determined by the convention \eqref{eq:section-action}. With the opposite convention for the action on sections, the same eigensection would be labelled by the
opposite weight. The affine-coordinate calculation in Section~\ref{sec:affine} provides an independent verification of this indexing and sign convention.
\end{remark}

Note that the fiber of the tangent bundle $T_{X_{\Delta}}$ over the identity of the dense torus
is naturally identified with \(N_{\C}\). We have the following standard Klyachko filtration of the tangent bundle.
\begin{proposition}[Klyachko\cite{Klyachko}]\label{prop:tangent-filtration}
Under the decreasing-filtration convention used here, the Klyachko
filtration of \(T_{X_\Delta}\) associated with a ray \(\rho\in\Delta(1)\) is
\begin{equation}\label{eq:tangent-filtration}
E^\rho(j) =
\begin{cases}
N_{\C},&j\leq0,\\[2pt]
\C v_\rho,&j=1,\\[2pt]
0,&j\geq2.
\end{cases}
\end{equation}
\end{proposition}
\begin{proof}
Choose a maximal cone \(\sigma\in\Delta(n)\) containing \(\rho\). Since \(\Delta\) is smooth, the primitive ray generators
\[
v_1,\dots,v_n
\]
of \(\sigma\) form a basis of \(N\). Let
\[
u_1,\dots,u_n
\]
be the dual basis of \(M\), and put
\[
z_j=\chi^{u_j}.
\]
Then
\[
U_\sigma\cong\C^n, \qquad \rho_N(v_j) = z_j\frac{\partial}{\partial z_j}.
\]
Equivalently,
\[
\frac{\partial}{\partial z_j} = \chi^{-u_j}\rho_N(v_j).
\]
The corresponding equivariant splitting shows that the line \(\C v_j\) occurs at filtration index \(1\), while the remaining directions occur at index \(0\). This gives \eqref{eq:tangent-filtration}.
\end{proof}

Klyachko filtrations are compatible with the usual functorial operations on equivariant vector bundles; see \cites{Gonzalez,Klyachko}.
Specializing to exterior powers \(\bigwedge^kT_{X_\Delta}\) of the tangent bundle \(T_{X_\Delta}\) gives rise to the following
\begin{proposition}[Standard exterior-power filtration]\label{prop:exterior-filtration}
For all \(1\leq k\leq n\), the Klyachko filtration of \(\bigwedge^kT_{X_\Delta}\) associated with the ray \(\rho\) is
\[
E_k^\rho(j) =
\begin{cases}
\bigwedge^kN_{\C},&j\leq0,\\[4pt]
v_\rho\wedge\bigwedge^{k-1}N_{\C},&j=1,\\[4pt]
0,&j\geq2.
\end{cases}
\]
\end{proposition}
\begin{proof}
Note that the filtration \eqref{eq:tangent-filtration} has only two nonzero graded levels. Its positive graded part is the one-dimensional subspace \(\C v_\rho\). A nonzero exterior monomial can contain
\(v_\rho\) at most once. Therefore the largest possible filtration index in \(\bigwedge^kN_{\C}\) is \(1\).

The index-\(1\) filtered subspace consists precisely of exterior products containing \(v_\rho\), namely
\[
v_\rho\wedge\bigwedge\nolimits^{k-1}N_{\C}.
\]
\end{proof}

\begin{remark}\label{rem:k-zero}
For \(k=0\), the relevant bundle is
\[
\bigwedge\nolimits^0T_{X_\Delta} = \mathcal{O}_{X_\Delta}.
\]
Its filtration is
\[
\C^\rho(j) =
\begin{cases}
\C,&j\leq0,\\
0,&j\geq1.
\end{cases}
\]
Thus Proposition~\ref{prop:exterior-filtration} is intended only for \(k\geq1\).
\end{remark}

\section{The weight-space decomposition}\label{sec:weights}
In this section, we prove the main theorem via Klyachko filtrations.

Fix \(u\in M\) and \(1\leq k\leq n\). By Theorem~\ref{thm:klyachko} and Proposition~\ref{prop:exterior-filtration},
\[
H^0\!\left(X_\Delta,\bigwedge\nolimits^kT_{X_\Delta} \right)_{-u}
\cong  \bigcap_{\rho\in\Delta(1)} E_k^\rho\!\left(-\Pair{u}{v_\rho}\right),
\]
where
\begin{equation}\label{eq:ray-condition}
E_k^\rho\!\left(-\Pair{u}{v_\rho}\right) =
\begin{cases}
\bigwedge^kN_{\CC},
&\Pair{u}{v_\rho}\geq0,\\[4pt]
v_\rho\wedge\bigwedge^{k-1}N_{\CC},
&\Pair{u}{v_\rho}=-1,\\[4pt]
0,
&\Pair{u}{v_\rho}\leq-2.
\end{cases}
\end{equation}
It follows that $H^0\left(X_\Delta,\wedge^k T_{X_\Delta} \right)_{-u}$ is nonzero only if
\[
u\in P_\Delta\cap M.
\]
For \(u\in P_\Delta\), set
\[
i(u)=\dim_{\CC}L_u.
\]
Consider a subspace of $\wedge^k N_\C$ defined by
\[
N_u^k =
\begin{cases}
\displaystyle
\bigwedge\nolimits^{i(u)}L_u
\wedge
\bigwedge\nolimits^{k-i(u)}N_{\CC},
&i(u)\leq k,\\[6pt]
0,&i(u)>k.
\end{cases}
\]

\begin{proposition}\label{prop:individual-weight}
For every \(u\in M\) and \(1\leq k\leq n\),
\[
H^0\!\left( X_\Delta,\bigwedge\nolimits^kT_{X_\Delta} \right)_{-u}
=
\begin{cases}
\chi^u\rho_N(N_u^k),
&u\in P_\Delta\cap M,\\[4pt]
0,&u\notin P_\Delta.
\end{cases}
\]
\end{proposition}

\begin{proof}
If \(u\notin P_\Delta\), then there is a ray \(\rho\) such that
\[
\Pair{u}{v_\rho}<-1.
\]
Since the pairing is integral,
\[
\Pair{u}{v_\rho}\leq -2.
\]
The corresponding factor in \eqref{eq:ray-condition} is zero, so the weight space vanishes.

Suppose now that \(u\in P_\Delta\cap M\). Rays for which
\[
\Pair{u}{v_\rho}\geq 0
\]
impose no condition. Therefore,
\[
H^0\!\left(X_\Delta,\bigwedge\nolimits^kT_{X_\Delta}\right)_{-u} \cong
\bigcap_{\rho\in A(u)} \left(v_\rho\wedge \bigwedge\nolimits^{k-1}N_{\CC} \right).
\]

For \(\xi\in\bigwedge^kN_{\CC}\), the condition
\[
\xi \in v_\rho\wedge\bigwedge\nolimits^{k-1}N_{\CC}
\]
is equivalent to
\[
v_\rho\wedge\xi=0.
\]
Recall that
\[
  \operatorname{span}_\C\{v_\rho \mid \rho \in A(u)\} = L_u.
\]
It follows that
\[
\bigcap_{\rho\in A(u)} \left(v_\rho\wedge \bigwedge\nolimits^{k-1}N_{\CC} \right) \cong N_u^k.
\]
Multiplication by \(\chi^u\) produces the corresponding global sections, whose weight is \(-u\) by
\eqref{eq:weight-sign}.
\end{proof}

For \(0\leq i\leq n\), denote by
\[
S(\Delta, i) = \bigcup_{\substack{F\preceq P_\Delta\\
                  \codim F=i}} \bigl(\relint(F)\cap M\bigr),
\]
the collection of lattice points lying in the relative interior $\relint(F)$ of some face $F$ of codimension $i$, and put
\[
S_k(\Delta) = \bigcup_{i=0}^{k}S(\Delta,i).
\]
By Proposition~\ref{lem:interior-origin},
\[
S_0(\Delta) = S(\Delta, 0) = \{0\}.
\]
Now we are ready to prove our main theorem.
\begin{theorem}[Weight-space decomposition]\label{thm:main}
For every \(1\leq k\leq n\),
\begin{equation}\label{eq:main-decomposition}
H^0\!\left( X_\Delta,\bigwedge\nolimits^kT_{X_\Delta} \right)
= \bigoplus_{u\in S_k(\Delta)} \chi^u\rho_N(N_u^k).
\end{equation}
The summand indexed by \(u\) is the weight-\((-u)\) subspace.

If \(u\in S(\Delta,i)\), then
\begin{equation}\label{eq:weight-dimension}
\dim_{\CC}\chi^u\rho_N(N_u^k) = \binom{n-i}{k-i}.
\end{equation}
\end{theorem}

\begin{proof}
By Proposition~\ref{prop:individual-weight}, the component indexed
by \(u\) is nonzero precisely when
\[
u\in P_\Delta\cap M
\quad\text{and}\quad
i(u)\leq k.
\]
If \(u\in\relint(F)\), then
\[
i(u)=\codim F
\]
by Lemma~\ref{lem:normal-space}. Thus the nonzero indices are
exactly the elements of \(S_k(\Delta)\).

Now let \(u\in S(\Delta,i)\). Choose a complement
\[
N_{\CC}=L_u\oplus Q.
\]
Then \(\dim Q=n-i\), and wedge product induces an isomorphism
\[
\bigwedge\nolimits^iL_u
\otimes
\bigwedge\nolimits^{k-i}Q
\longrightarrow
N_u^k.
\]

Since \(\bigwedge^iL_u\) is one-dimensional,
\[
\dim_{\CC}N_u^k = \dim_{\CC}\bigwedge\nolimits^{k-i}Q = \binom{n-i}{k-i}.
\]
Since the exterior extension of \(\rho_N\) is injective, the multiplication by the nonzero
character \(\chi^u\) is also injective on the dense torus. Hence,
\[
\dim_{\CC}\chi^u\rho_N(N_u^k) = \dim_{\CC}N_u^k,
\]
which proves \eqref{eq:weight-dimension}.
\end{proof}

\begin{remark}\label{rem:main-provenance}
The weight-space decomposition in \eqref{eq:main-decomposition} was previously announced in two
unpublished preprints of the first author \cites{HongPoisson, HongPolyvector}. It was subsequently stated, with attribution to \cite{HongPolyvector}, and applied by Deng and Hong \cite{Deng-Hong 21}. In the present formulation, the weight labels follow the convention fixed in \eqref{eq:section-action}, so that \(\chi^u\) times an invariant polyvector field has weight \(-u\).
\end{remark}

As an immediate consequence, we obtain the following dimension formula:
\begin{corollary}\label{cor:dimension}
For \(1\leq k\leq n\),
\[
\dim_{\CC} H^0\!\left( X_\Delta,\bigwedge\nolimits^kT_{X_\Delta} \right)
= \sum_{i=0}^{k} \binom{n-i}{k-i} \sum_{\substack{F\preceq P_\Delta\\ \codim F=i}}
\#\bigl(\relint(F)\cap M\bigr).
\]
\end{corollary}
\begin{proof}
Note that the number of lattice points in $S_k(\Delta)$ is
\[
\#S_k(\Delta)=\sum_{i=0}^k\sum_{\substack{F\preceq P_\Delta\\ \codim F=i}}
\#\bigl(\relint(F)\cap M\bigr).
\]
Using~\eqref{eq:weight-dimension}, we obtain the desired formula.
\end{proof}

\begin{remark}
For smooth toric Fano varieties, the numerical dimension formula may also be obtained from the Bott-type formulas of Materov~\cite{Materov 02}, using
\[
\bigwedge\nolimits^kT_{X_\Delta} \cong
\Omega_{X_\Delta}^{n-k} \otimes \mathcal{O}_{X_\Delta}(-K_{X_\Delta}).
\]
The filtration calculation additionally retains the individual torus-weight spaces.
\end{remark}

\section{Formal equivariant character}\label{subsec:formal-character}
In this section, we give a generating-function reformulation of Theorem~\ref{thm:main}, which can compactly encode the multiplicities of all torus weights across all exterior degrees.

We start by recalling the standard formal-character notation for rational representations of an algebraic torus; see Jantzen~\cite{Jantzen} and Chriss-Ginzburg~\cite{ChrissGinzburg}.

Let \(V\) be a finite-dimensional rational representation of the algebraic torus \(T\).
Since an algebraic torus is diagonalizable, \(V\) admits a weight-space decomposition
\[
V=\bigoplus_{m\in M}V_m,
\]
where
\[
V_m =
\left\{
v\in V\ \middle|\
t\cdot v=\chi^m(t)v
\text{ for every }t\in T
\right\}.
\]
Let \(\Z[M]\) be the integral group ring of \(M\). It has formal basis elements
\[
\mathbf e^m,\qquad m\in M,
\]
with multiplication
\[
\mathbf e^m\mathbf e^{m'} = \mathbf e^{m+m'}.
\]

\begin{definition}\label{def:formal-character}
The formal \(T\)-character of \(V\) is
\[
\charac_T(V) = \sum_{m\in M} \bigl(\dim_{\C}V_m\bigr)\mathbf e^m \in \Z[M].
\]
\end{definition}

Equivalently, the representation ring of the torus is naturally identified with the group ring
\[
R(T) \cong \Z[M],
\]
and \(\charac_T(V)\) is the image of the class \([V]\in R(T)\) under this identification; see
\cites{ChrissGinzburg, Jantzen}.

\begin{proposition}[Formal equivariant generating function]\label{cor:character}
In the ring \(\Z[M][q]\), one has
\begin{equation}\label{eq:character}
\sum_{k=0}^{n}q^k\, \charac_T H^0\!\left( X_\Delta,\bigwedge\nolimits^kT_{X_\Delta} \right)
= \sum_{u\in P_\Delta\cap M} \mathbf e^{-u} q^{i(u)}(1+q)^{n-i(u)}.
\end{equation}
\end{proposition}

\begin{proof}
For brevity, put
\[
V_k = H^0\!\left( X_\Delta,\bigwedge\nolimits^kT_{X_\Delta} \right).
\]
By \eqref{eq:weight-dimension} in Theorem~\ref{thm:main}, we have
\[
\dim_{\CC}(V_k)_{-u} = \begin{cases}
\displaystyle
\binom{n-i(u)}{k-i(u)},
&i(u)\leq k,\\[8pt]
0,
&i(u)>k.
\end{cases}
\]
Consequently,
\[
\charac_T(V_k) = \sum_{\substack{u\in P_\Delta\cap M\\i(u)\leq k}}
\binom{n-i(u)}{k-i(u)} \mathbf e^{-u},
\]
which implies that the left-hand side of~\eqref{eq:character} is
\[
\sum_{k=0}^{n} q^k \sum_{\substack{u\in P_\Delta\cap M\\i(u)\leq k}}
\binom{n-i(u)}{k-i(u)} \mathbf e^{-u}.
\]
Since \(P_\Delta\) is compact, \(P_\Delta\cap M\) is finite.
Interchanging the two finite sums gives
\[
\sum_{u\in P_\Delta\cap M} \mathbf e^{-u} \sum_{k=i(u)}^n \binom{n-i(u)}{k-i(u)}q^k.
\]

Fix \(u\), and write \(i=i(u)\). With \(r=k-i\), the inner sum is
\[
\sum_{k=i}^n \binom{n-i}{k-i}q^k = q^i \sum_{r=0}^{n-i} \binom{n-i}{r}q^r = q^i(1+q)^{n-i},
\]
by the binomial theorem. Summing over \(u\) proves~\eqref{eq:character}.

For \(k=0\), only \(u=0\) contributes. Indeed,
\[
H^0(X_\Delta,\mathcal{O}_{X_\Delta})=\C,
\]
and \(0\) is the unique interior lattice point of \(P_\Delta\) by Proposition~\ref{lem:interior-origin}.
\end{proof}

\section{Affine regularity and cancellation of poles}\label{sec:affine}

In this section, we compare the filtration calculation in Section~\ref{sec:weights} with a local-coordinate calculation.

Let \(\sigma\in\Delta(n)\) be a smooth maximal cone with primitive generators
\[
v_1,\dots,v_n.
\]
Let \(u_1,\dots,u_n\) be the dual basis of \(M\), and put
\[
z_j=\chi^{u_j}.
\]
Then
\[
U_\sigma\cong\CC^n, \qquad \rho_N(v_j) = z_j\frac{\partial}{\partial z_j}.
\]

Write
\[
u=\sum_{j=1}^{n}m_ju_j.
\]
Then
\[
m_j=\Pair{u}{v_j}, \qquad \chi^u=z_1^{m_1}\cdots z_n^{m_n}.
\]

\begin{proposition}\label{prop:affine-regularity}
Let
\[
\xi\in\bigwedge\nolimits^kN_{\CC}, \qquad 1\leq k\leq n.
\]
The polyvector field
\[
s=\chi^u\rho_N(\xi)
\]
on the dense torus extends regularly to \(U_\sigma\) if and only if
\begin{compactenum}
\item \(m_j \geq -1\) for every \(j\);
\item whenever \(m_j=-1\), one has
\[
\xi \in v_j\wedge\bigwedge\nolimits^{k-1}N_{\CC}.
\]
\end{compactenum}
\end{proposition}
\begin{proof}
We expand \(\xi\) in the exterior basis determined by \(v_1,\dots,v_n\).
Note that a basis term
\[
v_{j_1}\wedge\cdots\wedge v_{j_k} \in \bigwedge\nolimits^k N_\C
\]
is mapped by $\rho_N$ to
\[
z_{j_1}\cdots z_{j_k} \frac{\partial}{\partial z_{j_1}} \wedge \cdots\wedge
\frac{\partial}{\partial z_{j_k}}.
\]
After the multiplication by \(\chi^u\), its coefficient is
\[
z_1^{m_1}\cdots z_n^{m_n} z_{j_1}\cdots z_{j_k}.
\]
The exponent of \(z_j\) can be increased by at most one because \(v_j\) can occur at most once in a nonzero exterior monomial. Therefore regularity requires \(m_j\geq-1\).

If \(m_j=-1\), every exterior monomial with nonzero coefficient in \(\xi\) must contain \(v_j\) in order to cancel the simple pole \(z_j^{-1}\). This is equivalent to
\[
\xi\in v_j\wedge\bigwedge\nolimits^{k-1}N_{\CC}.
\]
Conversely, these conditions make every coordinate coefficient regular.
\end{proof}

Since
\[
m_j=\Pair{u}{v_j},
\]
condition $(1)$ is the local version of the inequalities defining \(P_\Delta\)  in~\eqref{eq:polytope}, while condition $(2)$ is precisely the index-\(1\) condition in Proposition~\ref{prop:exterior-filtration}.

\begin{remark}
The equality
\[
\Pair{u}{v_\rho}=-1
\]
means that \(\chi^u\) has a simple pole along \(D_\rho\). Requiring the invariant exterior factor to contain \(v_\rho\) supplies the corresponding Euler vector field and cancels this pole. When several
independent divisors are active, all their independent normal directions must occur in the exterior factor.
\end{remark}

\section{Consequences and examples}\label{sec:consequences}

\subsection{Demazure roots}
Consider the \(k=1\) case. The origin $0$ contributes the invariant vector fields
\[
\rho_N(N_{\C}).
\]
A nonzero lattice point contributes only if it belongs to the relative interior of a facet. Such a lattice point \(u\) satisfies
\[
\Pair{u}{v_{\rho_u}}=-1
\]
for a unique ray \(\rho_u\), and
\[
\Pair{u}{v_\rho}\geq 0 \qquad \text{for }\rho\neq\rho_u.
\]
These lattice points are known as the Demazure roots~\cites{Cox, Demazure}.

We now apply Theorem~\ref{thm:main} with \(k=1\).
Note that only faces of codimension zero and one contribute. The origin gives the \(n\)-dimensional space of invariant vector fields, and every lattice point in the relative interior of a facet contributes one dimension.
We immediately recover the following result.
\begin{corollary}[Demazure]\label{cor:demazure}
There is a weight decomposition
\[
H^0(X_\Delta,T_{X_\Delta}) = \rho_N(N_{\CC}) \oplus \bigoplus_{u\in S(\Delta,1)}
\CC\, \chi^u\rho_N(v_{\rho_u}).
\]
The nonzero summand indexed by \(u\) has weight \(-u\).
Consequently,
\[
\dim_{\CC}H^0(X_\Delta,T_{X_\Delta}) = n+\#S(\Delta,1).
\]
\end{corollary}

\subsection{Anticanonical sections}
For \(k=n = \dim_\C X_{\Delta}\), every lattice point of \(P_\Delta\) contributes one dimension.
If \(u\) lies in a face of codimension \(i\), then by Equation~\eqref{eq:weight-dimension} we have
\[
\dim N_u^n = \binom{n-i}{n-i}=1.
\]
As a consequence, we recover the following
\begin{corollary}\label{cor:top-degree}
There is a decomposition
\[
H^0\!\left( X_\Delta,\bigwedge\nolimits^nT_{X_\Delta} \right) =
\bigoplus_{u\in P_\Delta\cap M} \chi^u\rho_N\!\left(\bigwedge\nolimits^nN_{\CC}\right).
\]
In particular,
\[
\dim_{\CC}H^0(X_\Delta,\bigwedge\nolimits^nT_{X_\Delta})=\#(P_\Delta\cap M).
\]
\end{corollary}
\begin{remark}
Since
\[
\bigwedge\nolimits^nT_{X_\Delta} \cong \mathcal{O}_{X_\Delta}(-K_{X_\Delta}),
\]
the result agrees with the standard lattice-point description of sections of a torus-invariant divisor
\cites{Cox, Fulton}.
\end{remark}

\subsection{The blow-up of \(\mathbb P^3\) at a torus-fixed point}\label{subsec:blowup-P3}

Consider the blow-up $X$ of $\mathbb{P}^3$ at the torus-fixed point $p=[1:0:0:0]$
\[
\pi \colon X = \operatorname{Bl}_{p}\mathbb P^3 \longrightarrow\mathbb P^3.
\]
Thus
\[
\pi^{-1}(p)\cong\mathbb P^2,
\]
and
\[
\pi \colon  X\setminus\pi^{-1}(p) \longrightarrow \mathbb P^3\setminus\{p\}
\]
is an isomorphism.

Let \(N=\Z^3\), with standard basis \(e_1,e_2,e_3\). The fan \(\Delta\) of \(X\) has primitive ray generators
\[
v_0=(-1,-1,-1),\qquad v_1=(1,0,0),\qquad v_2=(0,1,0),
\]
\[
v_3=(0,0,1),\qquad v_4=(1,1,1).
\]
The ray generated by \(v_4\) is introduced by the star subdivision of
\[
\operatorname{Cone}(v_1,v_2,v_3),
\]
which corresponds to blowing up the fixed point \(p\); see~\cite{Cox, Fulton}.

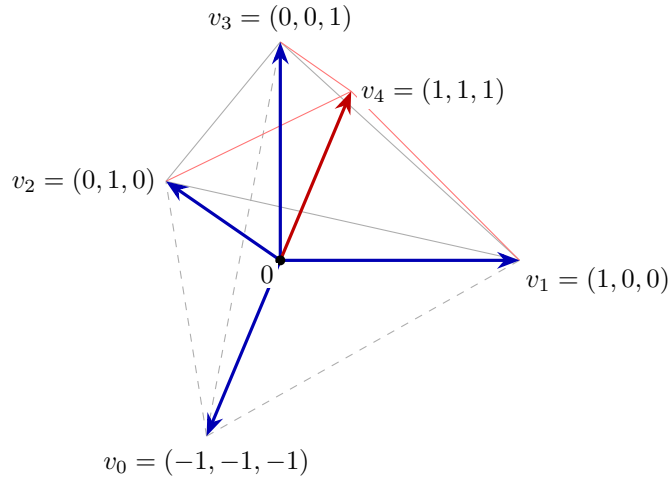
\begin{figure}[htbp]
\centering
\begin{tikzpicture}[
  x={(1.15cm,0cm)},
  y={(-0.55cm,0.38cm)},
  z={(0cm,1.05cm)},
  scale=1.25,
  >=Stealth,
  ray/.style={->,very thick,blue!70!black},
  newray/.style={->,very thick,red!75!black},
  coneedge/.style={thin,gray!65},
  labelstyle/.style={
    font=\small,
    fill=white,
    inner sep=1.5pt
  }
]

\coordinate (O) at (0,0,0);

\coordinate (Vzero) at (-1.30,-1.30,-1.30);
\coordinate (Vone)  at (2.20,0,0);
\coordinate (Vtwo)  at (0,2.20,0);
\coordinate (Vthree) at (0,0,2.20);
\coordinate (Vfour) at (1.25,1.25,1.25);

\draw[coneedge] (Vone)--(Vtwo);
\draw[coneedge] (Vone)--(Vthree);
\draw[coneedge] (Vtwo)--(Vthree);

\draw[coneedge,dashed] (Vzero)--(Vone);
\draw[coneedge,dashed] (Vzero)--(Vtwo);
\draw[coneedge,dashed] (Vzero)--(Vthree);

\draw[coneedge,red!55] (Vfour)--(Vone);
\draw[coneedge,red!55] (Vfour)--(Vtwo);
\draw[coneedge,red!55] (Vfour)--(Vthree);

\draw[ray] (O)--(Vzero);
\draw[ray] (O)--(Vone);
\draw[ray] (O)--(Vtwo);
\draw[ray] (O)--(Vthree);
\draw[newray] (O)--(Vfour);

\fill (O) circle (1.5pt);
\node[labelstyle,below left=1pt of O] {\(0\)};

\node[labelstyle,below right=1pt of Vone]
  {\(v_1=(1,0,0)\)};

\node[labelstyle,left=2pt of Vtwo]
  {\(v_2=(0,1,0)\)};

\node[labelstyle,above=2pt of Vthree]
  {\(v_3=(0,0,1)\)};

\node[labelstyle,right=2pt of Vfour]
  {\(v_4=(1,1,1)\)};

\node[labelstyle,below=3pt of Vzero]
  {\(v_0=(-1,-1,-1)\)};
\end{tikzpicture}

\caption{
A schematic projection of the fan of
\(\operatorname{Bl}_{p}\mathbb P^3\), where
\(p=[1:0:0:0]\).
}
\label{fig:blowup-fan}
\end{figure}

Identifying \(M_{\RR}\) with \(\RR^3\), the associated anticanonical polytope is
\begin{equation}\label{eq:blowup-polytope}
P_\Delta = \left\{ (x,y,z)\in\RR^3\ \middle|\  x,y,z\geq-1,\quad -1\leq x+y+z\leq1 \right\}.
\end{equation}
Indeed, the inequalities associated with \(v_1,v_2,v_3\) are
\[
x\geq-1,\qquad y\geq-1,\qquad z\geq-1,
\]
whereas those associated with \(v_0\) and \(v_4\) are
\[
x+y+z\leq1 \qquad\text{and}\qquad x+y+z\geq-1,
\]
respectively.

Geometrically, \(P_\Delta\) is a triangular prism bounded by the
two parallel facets
\[
x+y+z=-1 \qquad\text{and}\qquad x+y+z=1,
\]
and the three coordinate facets
\[
x=-1,\qquad y=-1,\qquad z=-1.
\]

\begin{figure}[htbp]
\centering
\begin{tikzpicture}[
  x={(1.05cm,0cm)},
  y={(-0.48cm,0.32cm)},
  z={(0cm,0.90cm)},
  scale=1.05,
  vertex/.style={circle,fill=black,inner sep=1.7pt},
  edge/.style={very thick,black!75},
  facetlower/.style={fill=blue!25,opacity=0.55},
  facetupper/.style={fill=orange!30,opacity=0.55},
  sidefacet/.style={fill=green!18,opacity=0.35},
  pointlabel/.style={
    font=\scriptsize,
    fill=white,
    inner sep=1.2pt,
    text opacity=1,
    fill opacity=0.85
  }
]

\coordinate (A) at (1,-1,-1);
\coordinate (B) at (-1,1,-1);
\coordinate (C) at (-1,-1,1);

\coordinate (D) at (3,-1,-1);
\coordinate (E) at (-1,3,-1);
\coordinate (F) at (-1,-1,3);

\coordinate (O) at (0,0,0);

\fill[sidefacet] (A)--(B)--(E)--(D)--cycle;
\fill[sidefacet] (B)--(C)--(F)--(E)--cycle;
\fill[sidefacet] (C)--(A)--(D)--(F)--cycle;

\fill[facetlower] (A)--(B)--(C)--cycle;
\fill[facetupper] (D)--(E)--(F)--cycle;

\draw[edge] (A)--(B);
\draw[edge] (B)--(C);
\draw[edge] (C)--(A);

\draw[edge] (D)--(E);
\draw[edge] (E)--(F);
\draw[edge] (F)--(D);

\draw[edge] (A)--(D);
\draw[edge] (B)--(E);
\draw[edge] (C)--(F);

\foreach \P in {A,B,C,D,E,F}
  \node[vertex] at (\P) {};

\fill[red!75!black] (O) circle (2pt);
\node[
  pointlabel,
  text=red!75!black,
  above right=1pt of O
]
{\(O=(0,0,0)\)};

\node[pointlabel,below right=2pt of A]
  {\((1,-1,-1)\)};

\node[pointlabel,below left=2pt of B]
  {\((-1,1,-1)\)};

\node[pointlabel,left=3pt of C]
  {\((-1,-1,1)\)};

\node[pointlabel,below right=2pt of D]
  {\((3,-1,-1)\)};

\node[pointlabel,above left=2pt of E]
  {\((-1,3,-1)\)};

\node[pointlabel,above=3pt of F]
  {\((-1,-1,3)\)};

\node[
  font=\footnotesize,
  text=blue!60!black,
  fill=white,
  inner sep=1.5pt
]
at (-0.25,-0.20,-0.55)
{\(x+y+z=-1\)};

\node[
  font=\footnotesize,
  text=orange!70!black,
  fill=white,
  inner sep=1.5pt
]
at (0.65,0.65,0.65)
{\(x+y+z=1\)};

\end{tikzpicture}

\caption{
A projected view of the anticanonical polytope
\(P_\Delta\) for \(\operatorname{Bl}_{p}\mathbb P^3\).
It is a triangular prism between the parallel facets
\(x+y+z=-1\) and \(x+y+z=1\).
The red point is the unique interior lattice point
\(O=(0,0,0)\).
}
\label{fig:blowup-polytope}
\end{figure}
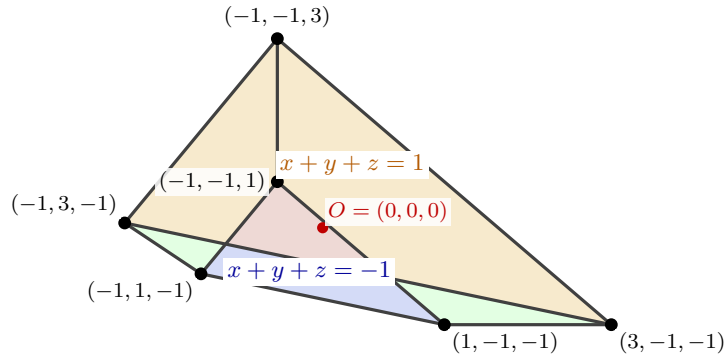

As shown in Figure~\ref{fig:blowup-polytope}, the six vertices of \(P_\Delta\) are
\[
(1,-1,-1),\qquad (-1,1,-1),\qquad (-1,-1,1),
\]
\[
(3,-1,-1),\qquad (-1,3,-1),\qquad (-1,-1,3).
\]

A direct enumeration gives the following face strata:
\[
S(\Delta,0)=\{(0,0,0)\},
\]
\[
\begin{split}
S(\Delta,1)=\{&
(0,1,-1),(0,-1,1),(1,0,-1),(-1,0,1),\\
&(1,-1,0),(-1,1,0),(1,0,0),(0,1,0),(0,0,1)
\},
\end{split}
\]
\[
\begin{split}
S(\Delta,2)=\{&
(0,2,-1),(2,-1,0),(-1,0,2),\\
&(2,0,-1),(-1,2,0),(0,-1,2),\\
&(2,-1,-1),(-1,2,-1),(-1,-1,2),\\
&(0,0,-1),(0,-1,0),(-1,0,0),\\
&(1,-1,1),(-1,1,1),(1,1,-1)
\},
\end{split}
\]
and
\[
\begin{split}
S(\Delta,3)=\{&
(1,-1,-1),(-1,1,-1),(-1,-1,1),\\
&(3,-1,-1),(-1,3,-1),(-1,-1,3)
\}.
\end{split}
\]

Thus,
\[
\#S(\Delta,0)=1,\qquad \#S(\Delta,1)=9,
\]
\[
\#S(\Delta,2)=15,\qquad \#S(\Delta,3)=6,
\]
and
\[
\#(P_\Delta\cap M)=31.
\]

Corollary~\ref{cor:dimension} gives
\[
\begin{split}
\dim_{\CC}H^0(X,T_X) &= 3\#S(\Delta,0)+\#S(\Delta,1)\\
&=3+9=12,
\end{split}
\]
\[
\begin{split}
\dim_{\CC}
H^0\!\left(X,\bigwedge\nolimits^2T_X\right) &= 3\#S(\Delta,0) +2\#S(\Delta,1) +\#S(\Delta,2)\\
&=3+18+15=36,
\end{split}
\]
and
\[
\dim_{\CC} H^0\!\left(X,\bigwedge\nolimits^3T_X\right) = \#(P_\Delta\cap M) =31.
\]

Finally, we write explicitly down a basis for each of these three vector spaces.

Note that the dense torus in \(\mathbb P^3\) is
\[
\widetilde T =
\left\{
[1:z_1:z_2:z_3]\in\mathbb P^3
\ \middle|\
z_1,z_2,z_3\in\CC^*
\right\}.
\]
Since \(p\notin\widetilde T\), the blow-up identifies the dense torus of \(X\) with \(\widetilde T\).

We first fix some notations: Put
\[
E_i=z_i\frac{\partial}{\partial z_i}, \qquad1\leq i\leq3,
\]
and
\[
E=E_1+E_2+E_3.
\]
Then
\[
\rho_N(v_i)=E_i,\qquad1\leq i\leq3,
\]
and
\[
\rho_N(v_0)=-E,\qquad \rho_N(v_4)=E.
\]

\subsubsection*{Vector fields}

The invariant vector fields are
\[
E_1,\qquad E_2,\qquad E_3.
\]
The nine nonzero root vectors are
\[
z_2\frac{\partial}{\partial z_3},\qquad
z_3\frac{\partial}{\partial z_2},\qquad
z_1\frac{\partial}{\partial z_3},
\]
\[
z_3\frac{\partial}{\partial z_1},\qquad
z_1\frac{\partial}{\partial z_2},\qquad
z_2\frac{\partial}{\partial z_1},
\]
together with
\[
z_1E,\qquad z_2E,\qquad z_3E.
\]
Together with the three invariant vector fields, these form a basis
of the \(12\)-dimensional space \(H^0(X,T_X)\).

\subsubsection*{Bivector fields}

The invariant bivector fields have basis
\[
E_1\wedge E_2,\qquad E_1\wedge E_3,\qquad E_2\wedge E_3.
\]
In coordinates these are
\[
z_1z_2 \frac{\partial}{\partial z_1} \wedge \frac{\partial}{\partial z_2},
\qquad
z_1z_3 \frac{\partial}{\partial z_1} \wedge \frac{\partial}{\partial z_3},
\qquad
z_2z_3 \frac{\partial}{\partial z_2} \wedge \frac{\partial}{\partial z_3}.
\]

For \(u\in S(\Delta,1)\), let \(r_u\) be the normal of the unique active facet. Then
\[
N_u^2=r_u\wedge N_{\CC}
\]
has dimension two. The active normals and generators are given in the following table:
\[
\begin{array}{c|c|c}
u&r_u& \text{generators} \\
\hline
(0,1,-1),(1,0,-1)&v_3& \chi^uE_3\wedge E_1,\; \chi^uE_3\wedge E_2; \\
(0,-1,1),(1,-1,0)&v_2& \chi^uE_2\wedge E_1,\; \chi^uE_2\wedge E_3; \\
(-1,0,1),(-1,1,0)&v_1& \chi^uE_1\wedge E_2,\; \chi^uE_1\wedge E_3; \\
(1,0,0),(0,1,0),(0,0,1)& v_0 & \chi^uE\wedge E_1,\; \chi^uE\wedge E_2.
\end{array}
\]
These provide the \(18\) facet-indexed bivectors.

For \(u\in S(\Delta,2)\), the normal space \(L_u\) is spanned by the two active facet normals \(r_u,s_u\). The corresponding one-dimensional weight space is generated by
\[
\chi^u\rho_N(r_u\wedge s_u).
\]
The active pairs and corresponding generators are summarized in the table below:
\[
\begin{array}{c|c|c}
u & \{r_u,s_u\} & \text{generator }\\
\hline
(0,2,-1),(2,0,-1),(1,1,-1)&\{v_3,v_0\} & \chi^uE_3\wedge E\\
(2,-1,0),(0,-1,2),(1,-1,1)&\{v_2,v_0\} & \chi^uE_2\wedge E\\
(-1,0,2),(-1,2,0),(-1,1,1)&\{v_1,v_0\} & \chi^uE_1\wedge E\\
(2,-1,-1)&\{v_2,v_3\}   &\chi^uE_2\wedge E_3\\
(-1,2,-1)&\{v_1,v_3\}  &\chi^uE_1\wedge E_3 \\
(-1,-1,2)&\{v_1,v_2\}  &\chi^uE_1\wedge E_2\\
(0,0,-1)&\{v_3,v_4\}  &\chi^uE_3\wedge E \\
(0,-1,0)&\{v_2,v_4\}  &\chi^uE_2\wedge E\\
(-1,0,0)&\{v_1,v_4\}  &\chi^uE_1\wedge E
\end{array}
\]
These contribute the remaining \(15\) bivectors. Hence, the total number of generators of $H^0(X, \wedge^2 T_X)$ is
\[
3+18+15=36,
\]
in agreement with the dimension formula.

\subsubsection*{Trivector fields}

For \(k=3\), every lattice point
\[
u=(u_1,u_2,u_3)\in P_\Delta\cap M
\]
contributes a one-dimensional weight space generated by
\[
\chi^uE_1\wedge E_2\wedge E_3.
\]
Since
\[
E_1\wedge E_2\wedge E_3 = z_1z_2z_3
\frac{\partial}{\partial z_1}
\wedge
\frac{\partial}{\partial z_2}
\wedge
\frac{\partial}{\partial z_3},
\]
this generator is
\[
z_1^{u_1+1}z_2^{u_2+1}z_3^{u_3+1}
\frac{\partial}{\partial z_1}
\wedge
\frac{\partial}{\partial z_2}
\wedge
\frac{\partial}{\partial z_3}.
\]

Put
\[
a_i=u_i+1,\qquad1\leq i\leq3.
\]
The defining inequalities of \(P_\Delta\) become
\[
a_i\geq0,
\qquad
2\leq a_1+a_2+a_3\leq4.
\]
Therefore the \(31\) generators are indexed by all monomials in
\(z_1,z_2,z_3\) of total degree \(d=2,3,4\).

The degree-two monomials are
\[
z_1^2,\quad z_2^2,\quad z_3^2,\quad
z_1z_2,\quad z_1z_3,\quad z_2z_3.
\]

The degree-three monomials are
\[
z_1^3,\quad z_2^3,\quad z_3^3,
\]
\[
z_1^2z_2,\quad z_1z_2^2,\quad z_1^2z_3,\quad z_1z_3^2,
\]
\[
z_2^2z_3,\quad z_2z_3^2,\quad z_1z_2z_3.
\]

The degree-four monomials are
\[
z_1^4,\quad z_2^4,\quad z_3^4,
\]
\[
z_1^3z_2,\quad z_1z_2^3,\quad
z_1^3z_3,\quad z_1z_3^3,
\]
\[
z_2^3z_3,\quad z_2z_3^3,
\]
\[
z_1^2z_2^2,\quad
z_1^2z_3^2,\quad
z_2^2z_3^2,
\]
\[
z_1^2z_2z_3,\quad
z_1z_2^2z_3,\quad
z_1z_2z_3^2.
\]

Multiplying every monomial in these lists by
\[
\frac{\partial}{\partial z_1}
\wedge
\frac{\partial}{\partial z_2}
\wedge
\frac{\partial}{\partial z_3}
\]
gives a basis of
\[
H^0\!\left(X,\bigwedge\nolimits^3T_X\right).
\]
The number of monomials is
\[
\binom{4}{2}+\binom{5}{2}+\binom{6}{2}
=
6+10+15
=
31.
\]

\begin{remark}
The example illustrates the face-codimension rule. The interior lattice point contributes invariant polyvectors. A facet-interior point imposes one normal factor, an edge-interior point imposes two,
and a vertex imposes three. The corresponding dimensions are
\[
\binom{3-i}{k-i},
\]
as stated in Theorem~\ref{thm:main}.
\end{remark}

{\bf Funding}

This work was supported by National Key R\&D Program of China [2022YFA1006200]; the Natural Science Funding of China [Grant NO. 12071241].

{\bf Acknowledgements}

We would like to thank Chuangqiang Hu, Yu Qiao and Ping Xu for helpful discussions and comments.
We are also grateful to anonymous referees for emphasizing the relationship with Klyachko's theory, for requesting a transparent account of the earlier preprints and the work of Deng--Hong, and for
pointing out the issues concerning the weight sign, the scope of the result, and the terminology for \(P_\Delta\). These comments led to a substantial reorganization of the manuscript.

{\bf Declaration of generative AI and AI-assisted technologies in the manuscript preparation process}

During the preparation and revision of this manuscript, the authors used AI to assist with English-language editing and the TikZ code of the two figures.
The authors take full responsibility for the content of the manuscript.

{\bf Conflict of interest}

We hereby declare that this work has no related financial or non-financial conflict of interests.

\end{document}